\documentclass[12pt]{amsart} 

\usepackage{amsmath,amssymb,amsthm,amsfonts,enumitem}
\usepackage{color}
\usepackage{graphicx,subcaption }
\usepackage{hyperref}
\usepackage{epstopdf}
\usepackage{bm}

\theoremstyle{plain} 
\newtheorem{thm}{Theorem}[section]
\newtheorem{lem}[thm]{Lemma}
\newtheorem{prop}[thm]{Proposition}

\theoremstyle{definition}
\newtheorem{defn}{Definition}[section]
\newtheorem{conj}{Conjecture}[section]

\theoremstyle{remark}

\begin{document}

\title{On minimal higher genus fillings}

\author{Gregory R. Chambers}
\address{Department of Mathematics, Rice University, Houston, TX}
\email{gchambers@rice.edu}
\date{\today}
\begin{abstract}
	In this article, we prove that if $(M,g)$ is a genus $G$ orientable surface with a single boundary component
	$S^1$, and if $(D,g_0)$ is a disc such that interior points are connected by unique geodesics and
	$$d_{(D,g_0)}(x,y) \geq d_{(M,g)}(x,y)$$
	for all $x,y \in \partial M = \partial D$, then $$(1 + \frac{2 G}{\pi})  \textrm{Area}(M,g) \geq \textrm{Area}(D,g_0).$$

\end{abstract}
\maketitle

\section{Introduction}

In ``Filling Riemannian manifolds" \cite{G1}, M. Gromov conjectured the following:
\begin{conj}[Gromov]
	\label{conj:Gromov}
	Suppose that $(M,g)$ is a Riemannian orientable surface with a single boundary
	component $S^1$ of length $2\pi$.  If, for every $x,y \in S^1$
	$$ d_{S^1} (x,y) = d_{(M,g)}(x,y),$$
	then the area of $(M,g)$ is no less than the area of the round hemisphere
	of intrinsic diameter $\pi$.
\end{conj}

Gromov also observed that if $M$ has genus $0$, this result can be obtained from Pu's inequality~\cite{Pu1},
which is a sharp inequality relating the length of the systole (shortest noncontractible loop) to the area of $\mathbb{R} P^2$
with a Riemannian metric.

In 2002, S. Ivanov~\cite{S1} gave a different proof of this result, in fact proving something more general:
\begin{thm}[S. Ivanov]
	\label{thm:ivanov}
	If $(D,g)$ and $(D,g_0)$ are two Riemannian discs with $$d_g(x,y) \geq d_{g_0}(x,y)$$ for all $x,y \in \partial D$,
	and if interior points of $(D,g_0)$ are connected by unique geodesics, then
	$$ \textrm{Area}(D,g) \geq \textrm{Area}(D,g_0).$$
\end{thm}
Conjecture~\ref*{conj:Gromov} is also known to be true if $(M,g)$ is a genus $1$ surface; this was proved using different methods 
in 2003 by V. Bangert, C. Croke, S. Ivanov, and M. Katz (see \cite{S2}).

Our main result generalizes Theorem~\ref*{thm:ivanov}:
\begin{thm}
	\label{thm:main}
	Suppose that $(D,g_0)$ is a Riemannian $2$-disc $D$, and that
	$(M,g)$ is an oriented surface with boundary $\partial M = \partial D = S^1$ and genus $G$. 
	If 
	\begin{enumerate}
		\item	For every two points $x$ and $y$ in the interior of $D$, there is a unique geodesic
			from $x$ to $y$.
		\item	For every $p$ and $q$ on $\partial D = \partial M$,
			$$d_{g_0}(p,q) \leq d_g(p,q).$$
	\end{enumerate}
	then $$Area(D,g_0) \leq ( 1 + \frac{2 G}{\pi}) Area(M,g).$$
\end{thm}

The problem of comparing the geometry of manifolds based on their boundary data has been studied in other contexts as well.  For example,
see \cite{BS1}, \cite{C1}, \cite{C2}, \cite{C3}, and \cite{G1L}.

\vspace{2mm}

\noindent {\bf Acknowledgments}
The author would like to thank Michael Wolf and for many useful conversations surrounding the ideas
in this article, as well for making numerous suggestions to earlier drafts.  He would also like to
thank Yevgeny Liokumovich for several discussions at the beginning of this project.  The author
was supported in part by NSF grant DMS-1906543.

\section{Overview of proof}

As mentioned in the introduction, the proof will build upon the methods developed by S. Ivanov in \cite{S1}.
Following this article, we begin by defining ``special distance functions" as follows.
Selecting a point $p$ on the boundary of $M$ (which is equal to the boundary of $D$), we define the function
$ f_p : M \rightarrow \mathbb{R} $
by $$ f_p(x) = \max_{q \in \partial M}( |pq|_0 - |xq| ).$$
Here, we are using the same notation as in \cite{S1}; $|\cdot|_0$ denotes the distance between points in $(D,g_0)$,
and $|\cdot|$ denotes the distance between points in $(M,g)$.  Note that the function $f_p$ is well-defined since $\partial M$
is compact.  Similarly, we define $\tilde{f}_p : D \rightarrow \mathbb{R}$ by
$$ \tilde{f}_p(x) = \max_{q \in \partial D} (|pq|_0 - |xq|_0).$$

The proof now works as follows.  We argue that, for a point $p \in \partial M = \partial D$, $f_p$ and $\tilde{f}_p$ agree on the boundary.
We then choose a sequence of distinct points $p_1, \dots, p_n$ on $\partial M = \partial D$, and consider the $2$-forms $\sum_{i=1}^n f_{p_i} \wedge f_{p_{i+1}}$
on $M$ and $\sum_{i=1}^n \tilde{f}_{p_1} \wedge \tilde{f}_{p_{i+1}}$ on $D$.  We show that both are exact, and so by Stokes' Theorem
their integrals can be written as integrals of a $1$-form over $\partial M$ and $\partial D$.  Since these boundaries agree and the functions agree on them
(for all $p_1, \dots, p_n$), these forms also are equal on the boundaries and so the integrals are equal as well.

We then argue that, as this sequence of points becomes dense in the boundary, the integral over the disc goes to $\pi \textrm{Area}(D,g_0)$, and
the integral over the genus $1$ surface is bounded by $\pi (1 + \frac{2 G}{\pi}) \textrm{Area}(M,g)$.  Since they agree (for any sequence of points) up to an arbitrarily small perturbation, we obtain the desired result.

\section{Properties of $f_p$ and $\tilde{f}_p$}

Throughout this section, the majority of the statements and their proofs follow their counterparts in \cite{S1}.
We begin with the following lemma:
\begin{lem}
	\label{lem:f_p_properties}
	Suppose that $p \in \partial M$, and $f_p$ is defined as above.  Then the following are true:
	\begin{enumerate}
		\item	$f_p$ is a nonexpanding function with respect to the metric $g$, that is, $|f_p(x) - f_p(y)| \leq |xy|$
			for all $x,y \in M$.
		\item $f_p$ is differentiable almost            everywhere.
	    \item	If $x \in \partial M$, then $f_p(x) = |px|_0$.
	\end{enumerate}
\end{lem}
\begin{proof}
	We will prove the three parts of the lemma in order:
	\begin{enumerate}
		\item	Fix $p \in \partial M$; $f_p$ is the supremum of functions $f_{p,q}(x) = - |xq| + |pq|_0$ over all $q \in \partial M$.
			We observe that $$|f_{p,q}(x) - f_{p,q}(y)| = | |xq| - |yq| | \leq |xy|,$$ where the last inequality
			is due to the reverse triangle inequality.  Thus, each $f_{p,q}$ is nonexpanding, and so the supremum
			over $q \in \partial M$ is also nonexpanding.
		\item	If $x \in \partial M$, then $|xq| \geq |xq|_0$ for all $q \in \partial M$ from the assumption about the boundary
			distances in $(D,g_0)$ and $(M,g)$.  Then
			$|pq|_0 - |xq| \leq |pq|_0 - |xq|_0 \leq |px|_0$ from the triangle inequality.  Furthermore, choosing $q = x$ we have
			$|px|_0 - |xx| = |px|_0$, completing the proof.
		\item This follows from the fact that $f_p$ is nonexpanding, and so is Lipschitz (with Lipschitz constant $1$).
	\end{enumerate}
\end{proof}

We also have the following lemma:
\begin{lem}
	\label{lem:f_g_disc_properties}
	Suppose that we fix a point $p \in \partial D$ (which is also equal to $\partial M$).  Defining $\tilde{f}_p$ as above, we have
	$\tilde{f}_p(x) = |xp|_0$ for all $x \in D$.
\end{lem}
\begin{proof}
	Fix $x$ in the interior of $D$, and let $q \in \partial D$ be the unique point on $\partial D$ where the geodesic from $p$ to $x$ intersects $\partial D$.
	Note that $q$ is unique by the assumption on the uniqueness of geodesics connecting interior points in $(D,g_0)$.

	Since this geodesic is unique, it is minimal, and $|pq|_0 = |px|_0 + |xq|_0$, so $\tilde{f}_p(x) \geq |pq|_0 - |qx|_0 = |px|_0$.  In addition,
	$\tilde{f}_p(x) \leq \tilde{f}_p(p) + |px|_0 = |px|_0$ from Part 1 of Lemma~\ref*{lem:f_p_properties} (note that $(D,g_0)$ satisfies the hypotheses of that lemma).
\end{proof}

\begin{lem}
	\label{lem:f_agreement}
	Suppose that $p \in \partial M = \partial D$, and $f_p$ and $\tilde{f}_p$ are defined as above.  Then $f_p(x) = \tilde{f}_p(x)$
	for every $x \in \partial D = \partial M$.
\end{lem}
\begin{proof}
	This lemma follows immediately from Lemma~\ref*{lem:f_p_properties} and Lemma~\ref*{lem:f_g_disc_properties}.
\end{proof}

\begin{defn}
	\label{defn:max}
	Suppose that $x$ is a point in the interior of $M$, and $p$ is a point on $\partial M$.  A point $q \in \partial M$
	is called a point of maximum if $f_p(x) = |pq|_0 - |xq|$.  It is called a \emph{nearest} point of maximum if, for
	every point of maximum $q'$, $|xq| \leq |xq'|.$
\end{defn}

\begin{lem}
	\label{lem:f_differentiable}
	Suppose that $p \in \partial M$, and $f_p$ is as above.  If
	$x$ is an interior point of $M$ and $f_p$ is differentiable at $x$, then $|df_p(x)| = 1$.  Moreover, if $q \in \partial M$
	is a point of maximum, and if $\gamma : [0, |xq|] \rightarrow M$ is the unit-speed shortest geodesic from $x$ to $q$,
	then $\textrm{grad} f_p(x) = \gamma'(0)$.  Furthermore, there is a unique point of maximum to $x$.
\end{lem}

\begin{proof}
	This proof will follow that of Lemma~1.2 in~\cite{S1}.  As in that proof, we will prove the statement for gradients
	instead of derivatives.

	Fix a point $x$ in the interior of $M$.  Suppose that $p \in \partial M$, and $f_p$ is differentiable at $x$, $q$ is a point of maximum
	for $p$, and the curve $\gamma : [0, |xq|] \rightarrow M$ is a unit-speed shortest curve connecting $x$ to $q$ in $(M,g)$;
	$\gamma(0) = x$, $\gamma(|xq|) = q$, and $|\gamma(t)\gamma(t')| = |t - t'|$ for all $t,t' \in [0,|xq|]$.  Since $x$ is in the interior
	of $M$, an initial arc of $\gamma$ is contained in the interior of $M$ and so is a geodesic.  As a result, $\gamma$ is differentiable
	at $0$, and $|\gamma'(0)| = 1$.  Since $f_p$ is nonexpanding and $f_p(q) = |pq|_0 = f_p(x) + |xq|$, $f_p(\gamma(t)) = f_p(x) + t$ for all $t \in [0,|xq|]$.

	From this observation that $f_p$ grows at unit rate along $\gamma$, and since this is the maximal growth rate (since $f_p$ is nonexpanding),
	$\textrm{grad} f_p(x) = \gamma'(0)$, and so $| \textrm{grad} f_p(x) | = 1$.

	To prove that there is a unique point of maximum to $x$, suppose that $q_1$ and $q_2$ are both nearest points of maximum to $x$.  Then the shortest geodesic
	from $x$ to $q_1$ and from $x$ to $q_2$ must have the same gradient at $x$. This is because the gradient along each curve has magnitude $1$, which is the maximal
	rate of growth of the function.  Hence, both shortest curves start at $x$ with the same tangent vector, and so due to the uniqueness of geodesics,
	that geodesic must hit the boundary of $M$ at $q_1$ first, or at $q_2$ first.  Thus, since the distances from $q_1$ to $x$ and $q_2$ to $x$ are equal,
	$q_1 = q_2$.
\end{proof}

We now investigate what happens if we have a finite set of points on the boundary of $M$.  For the remainder of the article, whenever we consider a sequence of points $\{ p_i \}$ on $\partial M$ or $\partial D$, we will assume that they are all distinct.  Furthermore, whenever we consider a sequence
of points $\{ q_i \}$ on $\partial M$ or $\partial D$ formed 
by taking the nearest points of maximum of unique $\{ p_i \}$ with respect to an interior point $x$, we will assume that
$\{ q_i \}$ are all unique as well.  This may require
a small perturbation of the points $\{ p_i \}$; this
does not affect the validity of the proofs throughout this article.

\begin{lem}
	\label{lem:order_of_points}
	Suppose that $p_1, \dots, p_n$ are points on $\partial M$ in counterclockwise order.  In addition, suppose that $x$ is a point in the interior of $M$,
	and that $f_{p_1}, \dots, f_{p_n}$ are all differentiable at $x$.  Let $q_1, \dots, q_n$ be the respective nearest points of maximum of
	$p_1, \dots, p_n$ with respect to $x$.  If $p_1, \dots, p_n$ and $q_1, \dots, q_n$ are all distinct, then $q_1, \dots, q_n$ are also in counterclockwise order
	on $\partial M$.
\end{lem}

To prove this lemma, we may assume that $n = 3$ without a loss in generality.  This is because the cyclic ordering of
a collection of points on $S^1$ is determined by the cyclic ordering of all triplets.  For the next lemma, assume that $p_1, p_2,$ and $p_3$ and
$q_1, q_2$, and $q_3$ be as in the hypotheses of Lemma~\ref*{lem:order_of_points}.

\begin{lem}
	\label{lem:separate}
	Fix $x$ in the interior of $M$, and suppose that $p_1, p_2,$ and $p_3$ and $q_1, q_2,$ and $q_3$ are points on $\partial M = \partial D$
	so that $q_i$ satisfies $f_{p_i}(x) = |xq_i| + |p_i q_i|_0$, $f_{p_i}$ is differentiable at $x$ for all $i$, and $q_i \neq p_i$ for all $i$.  If $i \neq j$, the pair $\{ p_i, q_j \}$
	does not separate the pair $\{ p_j, q_i \}$.  By this we mean that both $p_j$ and $q_i$ lie in the same component of
	$S^1 \setminus \{ p_i, p_j \}$.
\end{lem}
\begin{proof}
	Suppose that two pairs did separate; without loss of generality we may assume that $i = 1$ and $j = 2$.  Let $\alpha$ be a minimizing geodesic
	from $p_1$ to $q_2$ in $(D,g_0)$, and let $\beta$ be a minimizing geodesic from $p_2$ to $q_1$ in $(D,g_0)$.  Since $D$ is a disc, and the two pairs
	separate, $\alpha$ and $\beta$ must intersect at a point $z \in D$.  Thus, we have
	$$ |p_1 q_2|_0 + |p_2 q_1|_0 = |p_1 z|_0 + |z q_2|_0 + |p_2 z|_0 + |z q_1|_0 \geq |p_1 q_1|_0 + |p_2 q_2|_0,$$
	where the last inequality is due to the triangle inequality.  As a result,
	\begin{align*}
		|p_1 q_2|_0 - |x q_2| + |p_2 q_1|_0 - |x q_1| & \geq |p_1 q_1|_0 - |x q_1| + |p_2 q_2|_0 - |x q_2| \\
									  &  = f_{p_1}(x) + f_{p_2}(x).\\
	\end{align*}

	In addition, we observe that $|p_1 q_2|_0 - |x q_2| \leq f_{p_1}(x)$ and $|p_2 q_1|_0 - |x q_1| \leq f_{p_2}(x)$ from the definition
	of $f_{p_1}$ and $f_{p_2}$.  Thus, all of the above inequalities must be equalities, and so $q_1$ and $q_2$ are points of maximum for both
	$p_1$ and $p_2$.
\end{proof}

We use Lemma~\ref*{lem:separate} to prove Lemma~\ref*{lem:order_of_points}; this proof is identical to the second half of the proof
of Lemma~1.2 from \cite{S1} on page 5.

\section{Oriented triangles and Proof of Theorem~\ref*{thm:main}}

\begin{defn}
	\label{defn:oriented}
	Suppose that $a$ and $b$ are points on $S^1 \subset \mathbb{R}^2$.  The triangle $\Delta_{ab}$ is the one formed
	by the segment from $a$ to $b$, the segment from $a$ to $(0,0)$, and from $b$ to $(0,0)$.  Furthermore,
	$\Delta_{ab}$ carries an orientation.  If $a$ and $b$ are antipodal or equal, then we say that $\Delta_{ab}$ is neutrally
	oriented.  If not, then $a$ and $b$ separate $S^1$ into two closed arcs $\alpha$ and $\beta$ with nonempty disjoint
	interiors, and such that the length of $\alpha$ is less than the length of $\beta$ (with respect to the standard measure
	on $S^1$).  

	\noindent {\bf Case 1:} If we parametrize $\alpha$ in a counterclockwise fashion, and if $\alpha$ starts at $a$ and ends at $b$,
	then the triangle is positively oriented.

	\noindent {\bf Case 2:} If we parametrize $\alpha$ in a counterclockwise fashion, and if $\alpha$ starts at $b$ and ends at $a$,
	then the triangle is negatively oriented.

	The \emph{oriented area} of $\Delta_{ab}$ is the area of the triangle if the orientation is positive, and the negative
	of the area of the triangle if the orientation is negative.  If $\Delta_{ab}$ is neutrally oriented, then its area is $0$ and so we define
	its oriented area to be $0$.  With a slight abuse of notation, we will denote this as
	$\textrm{Area}(\Delta_{ab})$.
\end{defn}

\begin{lem}
	\label{lem:area_sequence_specific}
	Suppose that we have a sequence of ordered collections of points $\mathcal{P}_1, \dots, \mathcal{P}_k, \dots$ from $S^1$,
	and define $|\mathcal{P}_k|$ to be 
	$$ \sum_{i=1}^n \textrm{Area}(\Delta_{p_i p_{i+1}})$$
	where $p_1, \dots, p_n$ are the points (in counterclockwise order) that comprise $\mathcal{P}_k$ (and $p_{n+1} = p_n$).  In addition, for every $k$, let $\ell(k)$ be the length of the \emph{longest} segment of $S^1 \setminus \mathcal{P}_k$
	(using the standard measure on $S^1$).  Note that $\ell(k) > 0$ for all $k$ (since each collection of points is finite).
	
	If $\lim_{k \rightarrow \infty} \ell(k) = 0$, then
	$$ \lim_{k \rightarrow \infty} |\mathcal{P}_k| = \pi.$$
\end{lem}
\begin{proof}
	The proof of this fact is contained in the proof of Lemma~2.1 from~\cite{S1}; the limit
	is the area of a unit disc, which is $\pi$.  Note that
	for $k$ sufficiently large, each triangle formed from the
	collection $\mathcal{P}_k$ is positively oriented.
\end{proof}

\begin{prop}
	\label{prop:points_oriented_surface}
	Suppose that $p_1, \dots, p_n$ are distinct counterclockwise points on $\partial M$, and let $x$ in the interior of $M$ such that $f_{p_1}, \dots, f_{p_n}$
	are all differentiable at $x$.  In addition, let $q_1, \dots, q_n$ be points on $\partial M$ which are nearest points of minima to $x$
	with respect to $p_1, \dots, p_n$; suppose that they are distinct - by Lemma~\ref*{lem:order_of_points} they are oriented counterclockwise.  In addition, we have that
	there are unique length minimizing geodesics $\gamma_1, \dots, \gamma_n$ from $x$ to $q_1, \dots, q_n$ respectively.

	 If $\mathcal{P} = \{ v_1, \dots, v_n \}$ are the tangent vectors, then $$\sum_{i=1}^n \textrm{Area}(\Delta_{v_i v_{i+1}}) \leq \pi ( 1 + \frac{2 G}{\pi}).$$
\end{prop}

\begin{lem}
	\label{lem:points_disc}
	Suppose that $q_1, \dots, q_n$ are distinct points on $\partial D$ which are oriented in counterclockwise order, and suppose that $x$ is a point on the interior
	of $D$.  If $\gamma_1, \dots, \gamma_n$ are disjoint simple curves such that $\gamma_i$ starts at $x$ and ends at $q_i$.  Then $q_1, \dots, q_n$ also
	are oriented in a counterclockwise fashion, as are the tangent vectors $\{ v_1, \dots, v_n \}$.
\end{lem}
\begin{proof}
	This follows from the topology of the disc; in particular, the Jordan Curve Theorem directly leads to the conclusion of this lemma.
\end{proof}

For now, we will leave the statement of the Proposition~\ref*{prop:points_oriented_surface} as is, and will postpone the proof to the next section.
We now prove Theorem~\ref*{thm:main} following the same technique as \cite{S1}.

\begin{defn}
	\label{defn:omega_n}
	We define the $2$-form $\omega_n$ on $\mathbb{R}^n$ by
	$$ \omega_n = \sum_{i=1}^n d x_i \wedge d x_{i+1}. $$
\end{defn}

\begin{lem}
	\label{lem:pullback_M}
	For every $n \geq 2$, define $$F_M : M \rightarrow \mathbb{R}^n$$ by $$F(x) = (f_{p_1}(x), \dots, f_{p_n}(x))$$.
	We then have $$\int_M F_M^* \omega_n \leq 2 \pi (1 + \frac{2G}{\pi}) \textrm{Area}(M,g).$$
\end{lem}
\begin{proof}
	Let $\eta$ be the $2$-form $F_M^* \omega_n$ on $M$; $\eta$ is measurable and so $\eta(x) = A(\eta) d A$, where $A$ is a measurable real-valued function which is differentiable almost everywhere,
	and $d A$ is the standard area form (which exists since $M$ is orientable).
	Observe that $\eta = F_M^* \omega_n = \sum_{i=1}^n d f_{p_i} \wedge f_{p_{i+1}}$.  Observe
	further that $1/2 A(d f_{p_i} \wedge d f_{p_{i+1}}$ is equal to the oriented area of the triangle $\Delta v_i v_{i+1}$
	where $v_i = d f_{p_i} (x)$ and $v_{i+1} = d f_{p_{i+1}} (x)$ (here, $v_i$ and $v_{i+1}$ are in
	in $T^*_x M$).  As such, $1/2 A(\eta)$ is equal to the sum of the oriented areas of the triangles $\Delta v_i v_{i+1}$
	with $i \in \{1, \dots, n \}$.

	Note that all $v_i$ are unit vectors of $T^*_x M$, and so are in $UT^*_x M$.  Then by Proposition~\ref*{prop:points_oriented_surface}
	and Lemma~\ref*{lem:points_disc} together imply that 
	$$1/2 A(\eta) \leq \pi (1 + \frac{2G}{\pi}). $$
	As a result,
	$$ \int_M \eta \leq \int_M A(\eta) d A \leq \pi ( 1 + \frac{2G}{\pi}) \textrm{Area}(M,g). $$
\end{proof}

\begin{lem}
	\label{lem:pullback_D}
	Suppose that $\{ \mathcal{P}_1, \dots \}$ is a sequence of collections of points on $\partial D$.  For each
	sequence $\mathcal{P}_k = \{ p_n, \dots, p_n \}$, define $$F_D : D \rightarrow \mathbb{R}^n$$ by $$F(x) = (\tilde{f}_{p_1}(x), \dots, \tilde{f}_{p_n}(x)).$$
	We then have $$ \lim_{n \rightarrow \infty} \int_D F_D^* \omega_n = 2 \pi \textrm{Area}(D,g_0).$$
\end{lem}
\begin{proof}
	In this case, by Lemma~\ref*{lem:f_p_properties} we have $\tilde{f}_p (x) = |px|_0$, $\tilde{f}_p$ is differentiable everywhere on the interior of $D$ and
	geodesics are unique, since $\textrm{grad} f_p(x)$ is the unit vector opposite to the tangent of the unique geodesic from $x$ to $p$.
	Thus, the map that sends $p$ to $d f_p(x)$ is then a homeomorphism from $S^1$ to $UT^*_x D$.  Let $\mathcal{P} = \{ p_i \}$
	is a sequence of counterclockwise points on $\partial D$, let $\delta(\mathcal{P})$ be the length of the longest segment of $S^1 \setminus \{ p_i \}$.
	
	Let us choose a collection of points on $\partial D$ $\{ \mathcal{P}_1, \mathcal{P}_2, \dots \}$ so that $\delta(\mathcal{P}_i) \rightarrow 0$
	as $i \rightarrow \infty$.  As in the proof of Lemma~\ref*{lem:pullback_M}, for a given $\mathcal{P}_n$ the corresponding
	points $\{ v_i \}$ on $UT^*_x D$ are the images of the points under the homeomorphism from $\partial D = S^1 \rightarrow UT^*_x D$.

	Since it is a homeomorphism, if $\mathcal{V}_n$ is the set of points from $UT^*_x D$ that correspond to $\mathcal{P}_k$, the longest
	segment of $UT^*_x D \setminus \mathcal{V}_n$ also goes to $0$ as $k$ goes to $\infty$.  As per the proof of Lemma~\ref*{lem:pullback_M},
	we observe that if $\eta = F_D^* \omega_n$, then $\eta = A(\eta) d A$ and $1/2 A(\eta)$ is equal to the sum of the oriented areas of the
	triangles $\Delta v_i v_{i+1}$.  Lemma~\ref*{lem:area_sequence_specific} implies that all of these triangles are positively oriented for sufficiently 
	large $n$, and the fact that the longest segment
	of $UT^*_x D \setminus \mathcal{V}_n$ goes to $0$ as $k$ goes to $\infty$ implies that, as $k$ goes to $\infty$, $1/2 A(\eta)$ goes to $\pi$,
	so $A(\eta)$ goes to $2 \pi$.
\end{proof}

\begin{proof}[Proof of Theorem~\ref*{thm:main}]
	Let $F_M$, $F_D$, and $\omega_n$ be defined as above.
	From Lemma~\ref*{lem:f_agreement}, if $x$ is a point on $\partial D = \partial M$, and if $p$ is another point on $\partial D = \partial M$, then 
	$f_p(x) = \tilde{f}_p(x) = |px|_0$.  Thus, $F_M$ and $F_D$ agree on $\partial M = \partial D$.  Since $\omega_n$ is a closed $2$-form on $\mathbb{R}^n$, it is exact, and so $\omega_n = d \mu_n$ for some $1$-form $\mu_n$.  Then $F_M^* \omega_n = F_M^* d \mu_n = d F_M^* \mu_n$
	and $F_D^* \omega_n = F_D^* d \mu_n = d F_D^* \mu_n$.  Using Stokes' Theorem and the fact that $F_M^* \mu_n = F_D^* \mu_n$ on
	$\partial M = \partial D$, we have
	\begin{align*}
		\int_M A(\omega_n) d A &= \int_M F_M^* \omega_n \\
			               &= \int_{\partial M} F_M^* \mu_n \\
				       &= \int_{\partial D} F_D^* \mu_n \\
				       &= \int_D F_D^* \omega_n \\
				       &= \int_D A(\omega_n) d A
	\end{align*}
	This is true for any collection $\mathcal{P}$ of points on $S^1 = \partial M = \partial D$.  For any such collection, by Lemma~\ref*{lem:pullback_M},
	$2 \pi \textrm{Area}(M,g) \geq \int_M A(\omega_n) d A$.  By Lemma~\ref*{lem:pullback_D}, there is a sequence of collections $\{ \mathcal{P}_k \}$
	of points on the boundary so that $\int_D A(\omega_n) d A \rightarrow 2 \pi \textrm{Area}(D,g_0)$ as $k \rightarrow \infty$.

	Combining these estimates and taking the limit as $k \rightarrow \infty$, we have
	$$ 2 \pi (1 + \frac{2 G}{\pi})\textrm{Area}(M,g) \geq 2 \pi \textrm{Area}(D,g_0).$$
	Dividing both sides by $2 \pi$ completes the proof.
\end{proof}

\section{Proof of Proposition~\ref*{prop:points_oriented_surface}}

We begin this section with defining and proving the existence of special sequences of curves
with respect to an interior point $x$ of $(M,g)$.

\begin{defn}
	\label{defn:bouquet}
	Suppose that $(M,g)$ is a genus $G$ surface, and $x$ is an interior point of $M$.  We say that a sequence of $2G$ curves
	$\eta_1, \dots, \eta_{2G}$ which start and end at $x$ form a \emph{bouquet} with respect to $x$ if the following are true:
	\begin{enumerate}
		\item	If we remove $\eta_1, \dots, \eta_{2G}$ from $(M,g)$, then we obtain a half-open annulus.  One
			of the boundary components is the original $S^1$ (which we will call $C_1$), and one is an open $S^1$ that comes from
			removing these curves (which we will call $C_2$).
		\item All $\{ \eta_i \}$ are disjoint except for their initial and final points.
		\item	Suppose that $q$ is a point on $\partial M$, and $\nu$ is a curve from $q$ to $x$ which is length minimizing.
			Then $\nu$ only intersects each $\eta_i$ at $x$, the initial point of $\nu$.
	\end{enumerate}
	This is depicted in Figure~\ref*{fig:basic_unfolding}.
\end{defn}

\begin{figure}
	\caption{A bouquet for the genus-1 surface}
	\centerline{\includegraphics[width=0.5\textwidth]{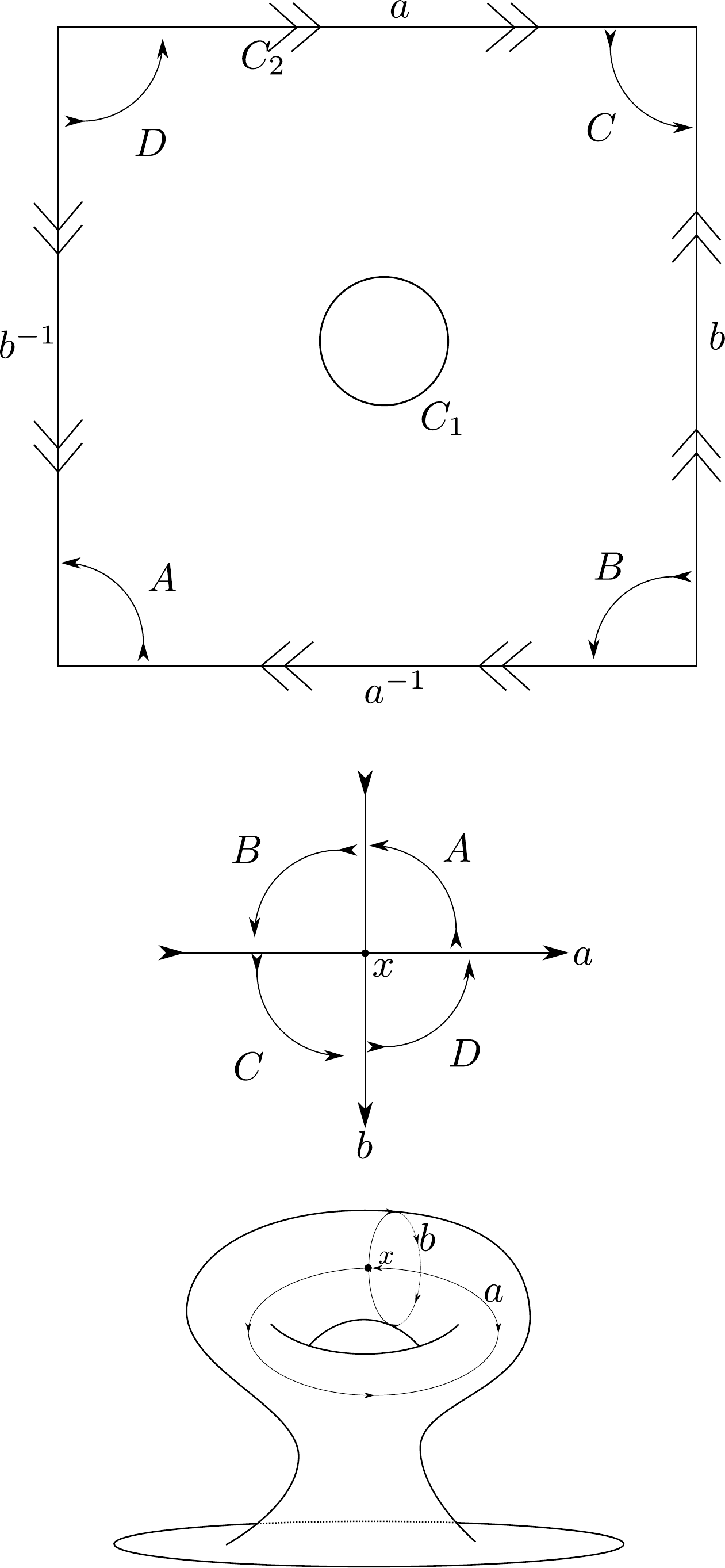}}
	\label{fig:basic_unfolding}
\end{figure}

\begin{lem}
	\label{lem:bouquet}
	Suppose that $(M,g)$ is a genus $G$ surface, and $x$ is an interior point of $M$.  Then there exists a bouquet $\eta_1, \dots \eta_{2G}$
	with respect to $x$.
\end{lem}
\begin{proof}
	We will prove the existence of such a sequence of curves by performing $G$ pairs of operations.  We will define the first pair of operations in detail,
	and then the others will be defined analogously.  Along the way, we will show that Properties 2 and 3 above (the no intersection properties)
	must be true for every intermediate sequence of curves.  After $i$ operations, we will obtain $2 i$ curves, which when removed from $M$ will result in a genus $G - i$ surface with
	two holes, one equal to the original $S^1$, and one equal to the open curve resulting from removing the $2 i$ curves.
	After $G$ operations, the result is an annulus with two boundary components, as desired.

	Let us now describe the process, starting from the first step.  To form $\eta_1$, we choose the shortest loop with endpoints equal to
	$x$ which is noncontractible and has minimal length.  Clearly, minimality guarantees that it is simple. Suppose that $q$ is a point on the boundary of $M$, and $\nu$ is a shortest curve
	from $x$ to $q$.  If the tangent vector of $\nu$ is parallel to the initial or final tangent vector of $\eta_1$, then the first segment of
	$\nu$ will be equal to $\eta_1$, which means that $\nu$ is not the shortest curve.
	
	Suppose that $y$ is another point of intersection;
	by the same argument as above, the tangent vector of $\nu$ and $\eta_1$ at the point $y$ must be linearly independent.  Then we can consider
	the two arcs from $x$ along $\eta_1$ to $y$, as well as the segment of $\nu$ from $x$ to $y$.  If either of the first two arcs are shorter than the arc along $\nu$, then $\nu$ is not the shortest curve.  If the arc along $\nu$ is shorter or equal in length to
	both of the original arcs, then we can replace one of the two arcs of $\eta_1$ with the arc of $\nu$ from $x$ to $y$ so that the result is
	noncontractible.  In addition, after smoothing out the resulting singularity, we can conclude that this new curve is both noncontractible and shorter than $\eta_1$, which is a contradiction.

	If we cut along $\eta_1$, we obtain a genus $G - 1$ surface with three boundary components.  One of these components corresponds to the original
	boundary component of the surface, and the other two come from cutting the original surface along $\eta_1$.  There are two points, $x_1$ and $x_2$,
	one on each of these components, which correspond to the original point
	$x$ before the unfolding.
	
	We define $\eta_2$ be the shortest curve that goes from the point $x_1$ on one of the boundary components to
	the point $x_2$ on the other boundary component; this curve does not touch either of the components with $x_1$ or $x_2$ since these components, as well as $\eta_2$, are all geodesics (and are uniquely determined by initial point and initial tangent vector).  It is also simple, as it must be length minimizing.  Suppose that $\nu$ is the shortest curve from $x$ to $q$, a point
	on the boundary of $M$.  Suppose that $\nu$ intersects $\eta_2$ at a point $y$ other than $x$.  The tangents of these curves at
	$y$ are linearly independent (otherwise $\nu$ would contain $\eta_2$ as a component, which would make it not the shortest curve).  Then we can follow the same procedure as before using $\eta_2$ instead of $\eta_1$; we can replace a segment of $\nu$ with a segment of $\eta_2$ to form a shorter competitor to $\nu$ (which is impossible), or replace a segment of $\eta_2$ with a segment of $\nu$ to form a better competitor to $\eta_2$, which is also impossible.
    This step is shown in Figure~\ref*{fig:bouquet}.

    Once we have found $\eta_1$ and $\eta_2$, if $G > 1$ we move on to form $\eta_3$ and $\eta_4$.  To form $\eta_3$, we do the following.  First, we cut along $\eta_1$ and $\eta_2$, the result being a surface with two boundary components, one ($B_1$) which corresponds to the original boundary of the surface, and one ($B_2$) which corresponds to cutting along $\eta_1$ and $\eta_2$.  $B_2$ contains $4$ points which correspond to $x$; $x_1, \dots, x_4$.  We then choose $\eta_3$ to be the shortest curve which starts at some $x_i$ and ends at some $x_j$ (potentially with $i = j$), and which is noncontractible relative to $B_2$.  $\eta_3$ can then be interpreted as a curve on the original surface $M$ starting and ending at $x$.
    
    Due to the uniqueness of geodesics, $\eta_3$ intersects $B_2$ only at its endpoints.  If $\nu$ is a length minimizing curve from $x$ to $y$, a point on $\partial M$, then it does not intersect $\eta_3$.  This argument is the
    same as that employed to show that this does not happen for $\eta_1$.
    
    As before, to form $\eta_4$, we cut along $\eta_1, \eta_2,$ and $\eta_3$.
    The result is a genus $G - 2$ surface with three boundary components,
    one which corresponds to the original boundary component, and two which come
    from cutting along $\eta_1, \eta_2,$ and $\eta_3$.  Each of these two components, $X$ and $Y$, contain points which correspond to $x$.
    
    We then define $\eta_4$ as the shortest curve which starts at a point corresponding to $x$ on $X$ to a point which corresponds to $x$ on $Y$.  Again, the uniqueness of geodesics implies that $\eta_4$ only intersects $\eta_1, \eta_2,$ and $\eta_3$ at its endpoints, and is simple (as it is length minimizing).  Furthermore, $\eta_4$ corresponds to a loop which starts and ends at $x$ on the original surface
    $M$.  If $\nu$ is a length minimizing curve which start at $x$ and ends at $y$ on $\partial M$, then it does not intersect $\eta_4$ except at $x$
    by the same argument that we employed above to show that the same
    property holds for $\eta_2$.
    
    If $G > 2$, then we execute the same procedure to form $\eta_5$ and $\eta_6$.  We continue to define $\eta_{2i + 1}$ and $\eta_{2i + 2}$
    until we obtain $\eta_1, \dots, \eta_{2G}$.  These curves
    satisfy all of the desired properties, completing the proof.
\end{proof}

\begin{figure}
	\caption{Building a bouquet for the genus-1 surface}
	\centering
	\includegraphics[width=0.5\textwidth]{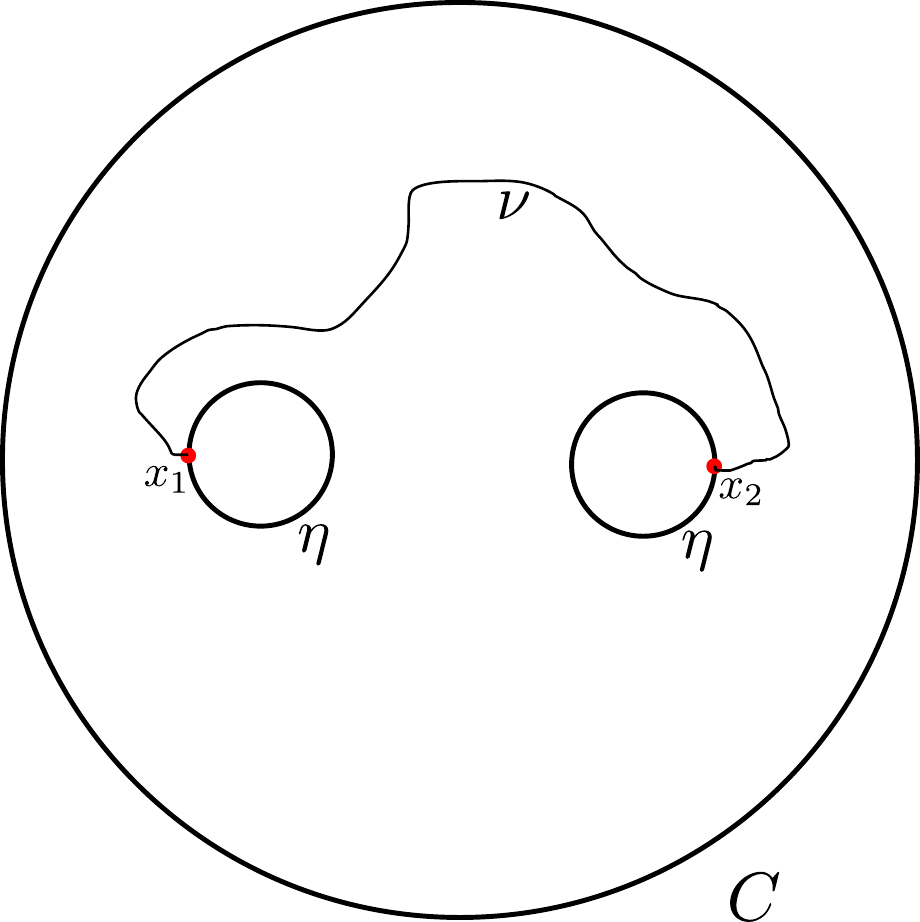}
	\label{fig:bouquet}
\end{figure}

\subsection{Proof of Proposition~\ref*{prop:points_oriented_surface}}

To complete the proof of Proposition~\ref*{prop:points_oriented_surface}, we will require a combinatorial result which is proved in the next section.  We will proceed assuming this result.

To begin, let $x$ be a point in the interior of $M$.   Suppose that $q_1, \dots, q_n$ are distinct points on $\partial M = S^1$
which are ordered counterclockwise on $\partial M$. Suppose further that there are unique minimizing geodesics $\alpha_1, \dots, \alpha_n$
which start at $x$ and end at $q_1, \dots, q_n$ (respectively).  We select a bouquet $\{ \beta_1, \beta_2, \dots, \beta_{2G - 1}, \beta_{2G} \}$
with respect to $x$; such a bouquet exists by Lemma~\ref*{lem:bouquet}.

We first observe that $\alpha_i$ and $\alpha_j$ with $i \neq j$
cannot be equal.  This follows from the fact that $q_i \neq q_j$.  Next, $\alpha_i$ and $\alpha_j$ only intersect at $x$.
If there is another point of intersection $y$ of $\alpha_i$ and $\alpha_j$, then since they are not equal this would imply that there
were multiple shortest curves from $x$ to $q_i$ and $q_j$, which is impossible.

Combining these observations with the definition of a bouquet, 
$\{ \alpha_i \}$ do not intersect $\{ \beta_i \}$
or each other (except at $x$).  As such, if we remove $\{ \beta_i \}$ from $M$, we then obtain an annulus with two boundary components.
Let $C_1$ be the component which is equal to $\partial M$, and $C_2$ is the new boundary component formed by removing $\{ \beta_1, \dots, \beta_{2G} \}$.
In this annulus $\mathcal{A}$, $\alpha_1, \dots, \alpha_n$ each starts at one of the $4G$ vertices of the polygon $C_2$ (corresponding
to the initial tangent vector of $\alpha_i$ at $x$), and ends at $q_1, \dots, q_n$ (respectively).  This is shown in Figure~\ref*{fig:annulus_point_setup}.

\begin{figure}
	\caption{The objects from the proof of Proposition~\ref*{prop:points_oriented_surface}}
	\centerline{\includegraphics[width=0.5\textwidth]{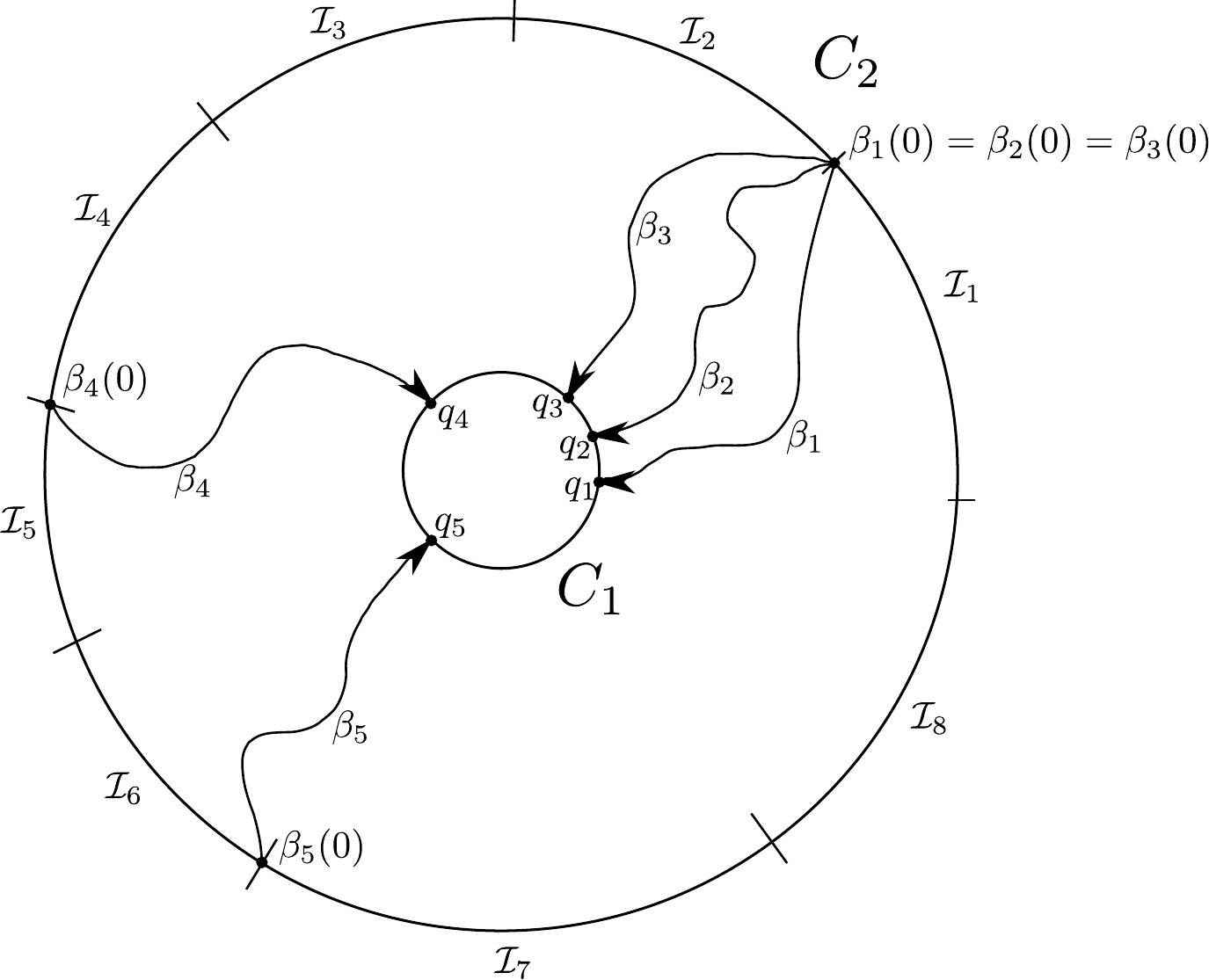}}
	\label{fig:annulus_point_setup}
\end{figure}

Let us be more specific.  The tangent vectors of both ends of each curve in $\{ \beta_i \}$ divide $UT_x$ up into segments;
each segment corresponds to one of the $4G$ vertices of the polygon with geodesic edges formed by cutting $M$ along $\{ \beta_i \}$
and unfolding the result.  Let $v_i$ be the unit vector in
$UT_x$ given by $\alpha_i'(0))$ (where $\alpha_i$ is unit-speed parametrization).  Then $\alpha_i$ starts at the point on the polygon
which corresponds to the interval of $UT_x$ in which $v_i$ lies (after the tangent vectors $\{ \beta_i'(0), -\beta_i'(1) \}$ are
all removed).  Furthermore, the order of all $\{v_i\}$ lying in the same segment corresponds to the order
of the tangent vectors of $\{ \alpha_i \}$ at their mutual starting point.

This setup satisfies the hypotheses of Proposition~\ref*{prop:total_angle_points} (stated and proved in the next section),
and so the total area of all of the triangles $\Delta_{v_1 v_2}, \Delta_{v_2 v_3}, \dots, \Delta_{v_{n-1} v_n}, \Delta_{v_n v_1}$
is at most $\pi$; this completes the proof.

\section{Combinatorial Results}

We will begin with the setup for Proposition~\ref*{prop:total_angle_points}, the main proposition
of this section.

Fix a counterclockwise sequence of distinct vectors $v_1, \dots, v_n$ on $S^1$, and fix
a positive integer $G \geq 1$.  Suppose that $\tau_1, \dots, \tau_{4G}$ are segments of $S^1$
whose union is $S^1$, have nonempty and disjoint interiors, and $v_i$ is in the interior of
some $\tau_j$.

Let $\mathcal{A}$ be the annulus $\{ (x,y) : 1 \leq ||(x,y)|| \leq 2 \} \subset \mathbb{R}^2$,
let $C_1$ be the inner boundary with radius $1$, and $C_2$ be the outer boundary with radius $2$
(see Figure~\ref*{fig:annulus_point_setup}).
Suppose that $p_1, \dots, p_{4G}$ are points on $C_2$ oriented counterclockwise.

Lastly, suppose that, for every $\tau_i$, we associate it with one of the points $p_j$ by
a bijective function $f : \{1, \dots, 4G \} \rightarrow \{1, \dots, 4G \}$.  We then consider curves
$\alpha_1, \dots, \alpha_n$ with the following properties:
\begin{enumerate}
	\item	Each $\alpha_i$ is simple.
	\item	Each $\alpha_i$ starts at $C_2$, and ends at $C_1$.
	\item	The initial point of $\alpha_i$ is one of the points $p_j$.  In particular, the initial point of $\alpha_i$ is equal to $p_{f(i)}$.
	\item   For each $p_i$, then curves which start at $p_i$ are exactly
	    $\alpha_{a_i}, \dots, \alpha_{b_i - 1}$ for some $b_i \geq a_i \geq 1$.  Note that this can be empty.
	\item	For every $\alpha_i$ and $\alpha_j$ with $i \neq j$, 
		the endpoints of $\alpha_i$ and $\alpha_j$ are distinct, and $\alpha_i$
		and $\alpha_j$ are disjoint except for possibly their initial points.
	\item   All $\{ \alpha) \}$ are length minimizing geodesics.
\end{enumerate}
For each $i$, let $v_i = \alpha_i'(0)$.

The main result concerns the order of the endpoints of $\{ \alpha_i \}$ as defined above.
We express this in terms of the total area of the triangles formed by points $\{ v_i \}$ as follows:
\begin{prop}
	\label{prop:total_angle_points}
	Suppose that both $f$ and $\{ \alpha_i \}$ are defined as above.
	Let $\{ q_i \}$ be the endpoints of $\{ \alpha_i \}$, respectively.  Then
	$$\sum_{i=1}^n \textrm{Area}(\Delta_{v_i v_{i+1}}) = \sum_{i=1}^n \textrm{Area}(\Delta_{\alpha_i'(0) \alpha_{i+1}'(0)}) \leq \pi (1 + \frac{2G}{\pi}).$$
\end{prop}

Before we prove Proposition~\ref*{prop:total_angle_points}, we will need an additional lemma.

\begin{lem}
	\label{lem:max_area}
	Suppose that $a$ and $b$ are points on $S^1$ (with unit radius).  Then
	$$ -\frac{1}{2} \leq \textrm{Area}(\Delta_{a b}) \leq \frac{1}{2}.$$
\end{lem}

\begin{proof}

	We need only prove the upper bound; the proof of the lower bound is analogous.
	Without loss of generality, we may assume that $b = (1,0)$ and $a$ lies on the
	upper half-circle.  Then the area of $\Delta_{a b}$ is equal to $\frac{\sin \theta}{2}$,
	where $\theta$ is the counterclockwise angle from $b$ to $a$ with $\theta \in [0, \pi]$.
	On this interval, $\frac{\sin \theta}{2}$ is maximized at $\theta = \pi/2$, and attains
	a maximum of $1/2$.

\end{proof}

We now move on to the proof of Proposition~\ref*{prop:total_angle_points}:
\begin{proof}

    Let us first consider each sequence $v_{a_i}, \dots, v_{b_i - 1}$ (with $v_j = \alpha_j'(0)$).  Since the curves are disjoint and are
    geodesics, the $\{ v_j \}$ are all distinct, and $v_{a_i}, \dots, v_{b_i - 1}$ are either all counterclockwise or all clockwise.
    Thus,
        $$\sum_{j = a_i}^{b_i-1} \textrm{Area}(\Delta v_j v_{j+1}) \leq \frac{\textrm{Length}(\tau_i)}{2},$$
    since all $v_{a_i}, \dots, v_{b_i - 1}$ lie in an interval of $S^1$ of length equal to that of $\tau_i$ (using the notation defined above).
    
    From this, we can conclude that 
        $$ \sum_{k = 1}^{4G} \sum_{j = a_k}^{b_k - 1} \textrm{Area}(\Delta v_j v_{j+1}) \leq 1/2 \sum_{k=1}^{4G} \textrm{Length}(\tau_k) = \pi.$$
    
    The series $$\sum_{i=1}^n \textrm{Area}(\Delta_{v_i v_{i+1}})$$ has
    exactly $4G$ more terms than $$\sum_{k = 1}^{4G} \sum_{j = a_k}^{b_k - 1} \textrm{Area}(\Delta v_j v_{j+1}).$$
    
    Each one of these terms, from Lemma~\ref*{lem:max_area}, is at most
    1/2, and so we get the bound
        $$\sum_{i=1}^n \textrm{Area}(\Delta_{v_i v_{i+1}}) \geq \pi + \frac{4 G}{2} = \pi(1 + \frac{2G}{\pi}),$$
    completing the proof.
\end{proof}

\bibliographystyle{amsplain}
\bibliography{bibliography}

\end{document}